\documentclass{amsart}
\usepackage[foot]{amsaddr}

\usepackage[latin1]{inputenc}
\usepackage{amssymb,amsmath,mathtools}
\usepackage{amsfonts}
\usepackage{marvosym}
\usepackage{enumitem}
\usepackage{booktabs}
\usepackage{ifthen}
\usepackage{tikz}
\usetikzlibrary{matrix}
\usepackage{url}
\usepackage{hyperref}
\usepackage{hyphenat}

\theoremstyle{plain}
\newtheorem{theorem}{Theorem}[section]
\newtheorem{proposition}[theorem]{Proposition}

\newtheorem{remark}[theorem]{Remark}

\theoremstyle{definition}

\newtheorem{problem}[theorem]{Problem}

\newcommand{\IN}{\ensuremath{\mathbb{N}}}
\newcommand{\emptyword}{\ensuremath{\varepsilon}}
\newcommand{\vect}[1]{\ensuremath{\mathbf{#1}}}
\DeclareMathOperator{\pr}{pr}
\DeclareMathOperator{\id}{id}

\DeclareMathOperator{\GF}{GF}

\begin{document}
\title{On associative operations on commutative integral domains}

\author{Erkko Lehtonen}

\author{Florian Starke}

\address
   {Technische Universit\"at Dresden \\
    Institut f\"ur Algebra \\
    01062 Dresden \\
    Germany}

\maketitle

\begin{abstract}
We describe the associative multilinear polynomial functions over commutative integral domains. This extends Marichal and Mathonet's result on infinite integral domains and provides a new proof of Andres's classification of two-element $n$\hyp{}semigroups.

\end{abstract}


\section{Introduction}

The classical notions of associativity and semigroup are generalized by $n$\hyp{}ary associativity and $n$\hyp{}semigroup.
Marichal and Mathonet described associative polynomial functions over infinite commutative integral domains \cite{MarMat-2011}.
We slightly modify this result to obtain a description of the associative
multilinear polynomial functions over arbitrary commutative integral domains.
As a special case, this gives a classification of the associative operations on a two\hyp{}element set (Boolean functions),
which was first established by Andres~\cite{Andres}.
We provide another, elementary proof of the result for Boolean functions, which is a streamlined version of the proof presented in \cite{Andres}.
Furthermore, we describe which $n$\hyp{}semigroups on a two\hyp{}element set are not derivable from any semigroup of smaller arity.


\section{Preliminaries}
\label{sec:preliminaries}

Throughout this paper, we denote the set of nonnegative integers by $\IN$.
Let $A$ be an arbitrary set.
An \emph{operation} on $A$ is a map $f \colon A^n \to A$ for some number $n \in \IN$, called the \emph{arity} of $f$.
An operation $f \colon A^n \to A$ is \emph{associative,} if for all $i, j \in \{0, \dots, n-1\}$, the equality
\begin{multline*}
f ( a_1, \dots, a_i, f ( a_{i+1}, \dots, a_{i+n} ), a_{i+n+1}, \dots,
a_{2n-1} ) \\
= f ( a_1, \dots, a_j, f ( a_{j+1}, \dots, a_{j+n} ), a_{j+n+1}, \dots,
a_{2n-1} )
\end{multline*}
holds for all $a_1, \dots, a_{2n-1} \in A$.
For $n = 2$, this condition is exactly the classical associative law.
Note that every unary operation is associative.
An algebra $(A; f)$ with a single $n$\hyp{}ary associative operation $f$ is called an \emph{$n$\hyp{}semigroup} or an \emph{$n$\hyp{}ary semigroup.}
Thus $2$\hyp{}semigroups are just the classical semigroups.

Given an $n$\hyp{}ary operation $f \colon A^n \to A$ and $\ell \in \IN$, we define the operation $f_\ell$ of arity $N(\ell) \coloneqq \ell (n - 1) + 1$ by the following recursion:
$f_0 \coloneqq \id_A$,
and for $\ell \geq 0$, define $f_{\ell+1} \colon A^{N(\ell+1)} \to A$ by the rule
\[
f_{\ell+1}(a_1, \dots, a_{N(\ell+1)}) \coloneqq
f_\ell(f(a_1, \dots, a_n), a_{n+1}, \dots, a_{N(\ell+1)}),
\]
for all $a_1, \dots, a_{N(\ell+1)} \in A$.
The operations $f_\ell$ are said to be \emph{derived} from $f$.
Note that $f_1 = f$.
In order to emphasize the arity of a derived operation, we will also write $f^{(N(\ell))}_\ell$ or simply $f^{(N(\ell))}$ for $f_\ell$.
It is easy to verify that if $f \colon A^n \to A$ is associative, then every operation derived from $f$ is associative.
Not every $n$\hyp{}ary associative operation arises in this way.
We say that an associative operation $f \colon A^n \to A$ is \emph{primitive,} if $f$ is not derivable from any associative operation $g \colon A^m \to A$ with $m < n$.


\section{Associative multilinear polynomial functions on integral domains}
\label{sec:related}

We would like to modify the following result so as to make it applicable for finite domains.

\begin{theorem}[{Marichal, Mathonet~\cite[Main Theorem]{MarMat-2011}}]
\label{thm:MarMat}
Let $R$ be an infinite commutative integral domain with identity and $n \geq 2$.
A polynomial function $p \colon R^n \to R$ is associative if and only if it is one of the following:
\begin{enumerate}[label={\upshape(\roman*)}]
\item $p(\mathbf{x}) = c$, where $c \in R$,
\item $p(\mathbf{x}) = x_1$,
\item $p(\mathbf{x}) = x_n$,
\item $p(\mathbf{x}) = c + \sum_{i=1}^n x_i$, where $c \in R$,
\item\label{thm:MarMat:v} $p(\mathbf{x}) = \sum_{i=1}^n \omega^{i-1} x_i$ \textup{(}if $n \geq 3$\textup{)}, where $\omega \in R \setminus \{1\}$ satisfies $\omega^{n-1} = 1$,
\item $p(\mathbf{x}) = -b + a \prod_{i=1}^n (x_i + b)$, where $a \in R \setminus \{0\}$ and $b$ is an element of the field of fractions of $R$ such that $ab^n - b \in R$ and $ab^k \in R$ for every $k \in \{1, \dots, n-1\}$.
\end{enumerate}
\end{theorem}

Marichal and Mathonet's proof of Theorem~\ref{thm:MarMat} starts with the observation that polynomials and polynomial functions over $R$ are in one\hyp{}to\hyp{}one correspondence.
Then it is shown in \cite[Proposition~2]{MarMat-2011} that for any associative polynomial function $p \colon R^n \to R$, the polynomial $p$ must be multilinear, i.e., no variable occurs with an exponent higher than $1$.
The remainder of the proof only relies on this multilinearity.
Thus, restricting ourselves to multilinear polynomials to begin with, we can also allow finite domains, and we are lead to the following result.

\begin{theorem}
\label{thm:gen-MarMat}
Let $R$ be a commutative integral domain with identity and $n \geq 2$.
A multilinear polynomial function $p \colon R^n \to R$ is associative if and only if it is of one of the forms prescribed in Theorem~\ref{thm:MarMat}.
\end{theorem}

Since every Boolean function is a multilinear polynomial function over $\GF(2)$,
a description of associative Boolean functions follows immediately from Theorem~\ref{thm:gen-MarMat} (note that item \ref{thm:MarMat:v} is void in this case).
On the other hand, the theorem fails to capture all associative functions over finite fields with at least three elements.
It is easy to provide examples of $n$\hyp{}ary associative operations that are not of any of the forms listed in Theorem~\ref{thm:MarMat}, such as
$n$\hyp{}semigroups derived from rectangular bands,
or operations of the form $(x_1, \dots, x_n) \mapsto \varphi(x_1)$, where $\varphi \colon A \to A$ is a nonconstant idempotent map distinct from $\id_A$.
This leads to an intriguing open problem.

\begin{problem}
Describe the associative operations on a finite set with at least three elements.
\end{problem}


\section{Elementary proof of the description of two\hyp{}element $n$\hyp{}semigroups}

It is well known that there are eight semigroups on the two\hyp{}element set $\{0,1\}$.
They are precisely the algebras with one of the following binary operations:
constant operations $c_0$, $c_1$,
projections $\pr^{(2)}_1$, $\pr^{(2)}_2$,
semigroup operations $\vee$, $\wedge$,
and group operations $+$, $\boxplus$.
These operations are defined in Table~\ref{table:semigroups}.

\begin{table}
\begin{center}
\begin{tabular}{rrrr}
\begin{tabular}[t]{r|cc}
$c_0$ & $0$ & $1$ \\
\hline
$0$ & $0$ & $0$ \\
$1$ & $0$ & $0$
\end{tabular}
&
\begin{tabular}[t]{r|cc}
$c_1$ & $0$ & $1$ \\
\hline
$0$ & $1$ & $1$ \\
$1$ & $1$ & $1$
\end{tabular}
&
\begin{tabular}[t]{r|cc}
$\pr^{(2)}_1$ & $0$ & $1$ \\
\hline
$0$ & $0$ & $0$ \\
$1$ & $1$ & $1$
\end{tabular}
&
\begin{tabular}[t]{r|cc}
$\pr^{(2)}_2$ & $0$ & $1$ \\
\hline
$0$ & $0$ & $1$ \\
$1$ & $0$ & $1$
\end{tabular}
\\
\\
\begin{tabular}[t]{r|cc}
$\vee$ & $0$ & $1$ \\
\hline
$0$ & $0$ & $1$ \\
$1$ & $1$ & $1$
\end{tabular}
&
\begin{tabular}[t]{r|cc}
$\wedge$ & $0$ & $1$ \\
\hline
$0$ & $0$ & $0$ \\
$1$ & $0$ & $1$
\end{tabular}
&
\begin{tabular}[t]{r|cc}
$+$ & $0$ & $1$ \\
\hline
$0$ & $0$ & $1$ \\
$1$ & $1$ & $0$
\end{tabular}
&
\begin{tabular}[t]{r|cc}
$\boxplus$ & $0$ & $1$ \\
\hline
$0$ & $1$ & $0$ \\
$1$ & $0$ & $1$
\end{tabular}
\end{tabular}
\end{center}

\bigskip
\caption{The semigroup operations on $\{0,1\}$.}
\label{table:semigroups}
\end{table}

For notational simplicity, in what follows, we view tuples over $\{0,1\}$ as words in the free monoid with two generators $0$ and $1$, and we will concatenate words by writing them one after the other.
We use the shorthand $a^n$ for a word comprising $n$ copies of $a$.
In particular, $a^0$ equals the empty word $\emptyword$, $a^1 = a$, and $a^2 = aa$.
Moreover, we denote an $n$\hyp{}ary operation by brackets, i.e., $(\cdot) \colon \{0,1\}^n \to \{0,1\}$, $a_1 a_2 \dots a_n \mapsto (a_1 a_2 \dots a_n)$.

\begin{theorem}[{Andres~\cite[Theorem~3.1]{Andres}}]
\label{thm:main}
For $n \geq 2$, an $n$\hyp{}ary operation on $\{0,1\}$ is associative if and only if it is one of the following:
$c^{(n)}_0$, $c^{(n)}_1$, $\pr^{(n)}_1$, $\pr^{(n)}_n$, $\vee^{(n)}$, $\wedge^{(n)}$, $+^{(n)}$, $\overline{+^{(n)}}$,
where $\overline{+^{(n)}}(a_1, \dots, a_n) \coloneqq +^{(n)}(a_1, \dots, a_n) + 1$.
\end{theorem}

\begin{proof}
It is clear that the operations specified in the statement are associative, because each one is derived from a binary associative operation, with the exception of $\overline{+^{(n)}}$ for odd $n$.
(We have $\boxplus^{(n)} = \overline{+^{(n)}}$ for even $n$, and $\boxplus^{(n)} = +^{(n)}$ for odd $n$.)
It is easy to verify that also $\overline{+^{(n)}}$ is associative.

In order to show necessity, assume that $(\cdot) \colon \{0,1\}^n \to \{0,1\}$ is associative.
We consider several cases and subcases.

\renewcommand{\descriptionlabel}[1]{\hspace*{\labelsep}\textsc{#1:}}
\setlist{parsep=1ex, itemsep=1ex, leftmargin=0cm, labelindent=0cm}

\begin{description}
\item[Case 1]
$(0^n) = 0$.
\begin{description}
\item[Case 1.1]
$(1 0^{n-1}) = 0$.
Then for all $\vect{u} \vect{v} \in \{0,1\}^{n-2}$,
\[
(\vect{u} 1 0 \vect{v})
= (\vect{u} 1 (0^n) \vect{v})
= (\vect{u} (1 0^{n-1}) 0 \vect{v})
= (\vect{u} 0 0 \vect{v}).
\]

It follows that $(\cdot)$ is completely determined by its values at tuples of the form
$0^{n-k} 1^k$ with $0 \leq k \leq n$.
More precisely, if $\vect{a} = \vect{u} 0 1^k$ for some $\vect{u} \in \{0,1\}^{n-k-1}$, then $(\vect{a}) = (0^{n-k} 1^k)$, because the value of $(\cdot)$ does not change if we change any $1$ followed by $0$ to $0$.

\begin{description}
\item[Case 1.1.1]
$(0^{n-1} 1) = 0$.
Similarly as above, we obtain for all $\vect{u} \vect{v} \in \{0,1\}^{n-2}$,
\[
(\vect{u} 0 1 \vect{v})
= (\vect{u} (0^n) 1 \vect{v})
= (\vect{u} 0 (0^{n-1} 1) \vect{v})
= (\vect{u} 0 0 \vect{v}).
\]
Thus the value of $(\cdot)$ does not change if we change any $1$ preceded by $0$ to $0$;
in particular, $(0^{n-k} 1^k) = (0^{n-k+1} 1^{k-1})$ for $0 < k < n$.
Consequently, $(\vect{a}) = 0$ for all $\vect{a} \in \{0,1\}^n \setminus \{1^n\}$.
It remains to consider the value of $(\cdot)$ at $1^n$.

\begin{description}
\item[Case 1.1.1.1]
$(1^n) = 0$.
Then $(\cdot) = c^{(n)}_0$.

\item[Case 1.1.1.2]
$(1^n) = 1$.
Then $(\cdot) = \wedge^{(n)}$.
\end{description}

\item[Case 1.1.2]
$(0^{n-1} 1) = 1$.
Then for all $\vect{u} \vect{v} \in \{0,1\}^{n-2}$,
\[
(\vect{u} 0 1 \vect{v})
= (\vect{u} (1 0^{n-1}) 1 \vect{v})
= (\vect{u} 1 (0^{n-1} 1) \vect{v})
= (\vect{u} 1 1 \vect{v}).
\]
Thus
$(0^{n-k} 1^k) = (0^{n-k-1} 1^{k+1})$ for $0 < k < n$,
and we have $(\cdot) = \pr^{(n)}_n$.
\end{description}

\item[Case 1.2]
$(1 0^{n-1}) = 1$.

\begin{description}
\item[Case 1.2.1]
$(1 1 0^{n-2}) = 0$.
Then for all $\vect{u} \vect{v} \in \{0,1\}^{n-2}$,
\[
(\vect{u} 0 0 \vect{v})
= (\vect{u} (1 1 0^{n-2}) 0 \vect{v})
= (\vect{u} 1 (1 0^{n-1}) \vect{v})
= (\vect{u} 1 1 \vect{v}).
\]

\begin{description}
\item[Case 1.2.1.1]
$(0 1 0^{n-2}) = 0$.
Then for all $\vect{u} \vect{v} \in \{0,1\}^{n-2}$,
\[
(\vect{u} 0 0 \vect{v})
= (\vect{u} (0 1 0^{n-2}) 0 \vect{v})
= (\vect{u} 0 (1 0^{n-1}) \vect{v})
= (\vect{u} 0 1 \vect{v}).
\]
Applying the above identities, we obtain
\[
0
= (0 0 0 0^{n-3})
= (0 0 1 0^{n-3})
= (1 1 1 0^{n-3})
= (1 0 0 0^{n-3})
= 1,
\]
a contradiction.
Thus, this case is not possible.

\item[Case 1.2.1.2]
$(0 1 0^{n-2}) = 1$.
Then for all $\vect{u} \vect{v} \in \{0,1\}^{n-2}$,
\[
(\vect{u} 0 1 \vect{v})
= (\vect{u} 0 (1 0^{n-1}) \vect{v})
= (\vect{u} (0 1 0^{n-2}) 0 \vect{v})
= (\vect{u} 1 0 \vect{v}).
\]
This means that $(\cdot)$ is symmetric, and the value of $(\cdot)$ at $\vect{a}$ depends only on the number of $1$'s in $\vect{a}$.
Together with the identity $(\vect{u} 0 0 \vect{v}) = (\vect{u} 1 1 \vect{v})$ established above, this implies that $(\vect{a})$ depends only on the parity of the number of $1$'s in $\vect{a}$.
Since $(0^n) = 0$ and $(1 0^{n-1}) = 1$, we have
$(\vect{a}) = 0$ if and only if the number of $1$'s in $\vect{a}$ is even,
in other words, $(\cdot) = +^{(n)}$.
\end{description}

\item[Case 1.2.2]
$(1 1 0^{n-2}) = 1$.
Then for all $\vect{u} \vect{v} \in \{0,1\}^{n-2}$,
\[
(\vect{u} 1 0 \vect{v})
= (\vect{u} (1 1 0^{n-2}) 0 \vect{v})
= (\vect{u} 1 (1 0^{n-1}) \vect{v})
= (\vect{u} 1 1 \vect{v}).
\]
It follows that $(\cdot)$ is completely determined by its values at tuples of the form
$0^k 1^{n-k}$ with $0 \leq k \leq n$.
More precisely, if $\vect{a} = 0^k 1 \vect{u}$ for some $\vect{u} \in \{0,1\}^{n-k-1}$, then $(\vect{a}) = (0^k 1^{n-k})$, because the value of $(\cdot)$ does not change if we change any $0$ preceded by $1$ to $1$.
In particular, for any $\vect{u} \in \{0,1\}^{n-1}$, we have $(1 \vect{u}) = (1^n) = (1 0^{n-1}) = 1$.

\begin{description}
\item[Case 1.2.2.1]
$(0 1^{n-1}) = 0$.
Then for all $\vect{u} \vect{v} \in \{0,1\}^{n-2}$,
\[
(\vect{u} 0 1 \vect{v})
= (\vect{u} 0 (1^{n-1} 0) \vect{v})
= (\vect{u} (0 1^{n-1}) 0 \vect{v})
= (\vect{u} 0 0 \vect{v}).
\]
Thus $(0^k 1^{n-k}) = (0^{k+1} 1^{n-k-1})$ for $0 < k < n$.
Consequently, $(0^k 1^{n-k}) = (0^n) = 0$ for $0 < k \leq n$.
Therefore $(\cdot) = \pr^{(n)}_1$.

\item[Case 1.2.2.2]
$(0 1^{n-1}) = 1$.
Then for all $\vect{u} \vect{v} \in \{0,1\}^{n-2}$,
\[
(\vect{u} 0 1 \vect{v})
= (\vect{u} 0 (1^n) \vect{v})
= (\vect{u} (0 1^{n-1}) 1 \vect{v})
= (\vect{u} 1 1 \vect{v}).
\]
Thus $(0^k 1^{n-k}) = (0^{k-1} 1^{n-k+1})$ for $0 < k < n$.
Consequently, $(0^k 1^{n-k}) = (1^n) = 1$ for $0 \leq k < n$.
Therefore $(\vect{a}) = 0$ if and only if $\vect{a} = 0^n$,
that is, $(\cdot) = \vee^{(n)}$.
\end{description}
\end{description}
\end{description}

\item[Case 2]
$(0^n) = 1$.
Then for all $\vect{u} \vect{v} \in \{0,1\}^{n-2}$,
\[
(\vect{u} 0 1 \vect{v})
= (\vect{u} 0 (0^n) \vect{v})
= (\vect{u} (0^n) 0 \vect{v})
= (\vect{u} 1 0 \vect{v}).
\]
Consequently, $(\cdot)$ is symmetric, and the value of $(\cdot)$ at $\vect{a}$ depends only on the number of $1$'s in $\vect{a}$.

\begin{description}
\item[Case 2.1]
$(1 0^{n-1}) = 0$.
Then for all $\vect{u} \vect{v} \in \{0,1\}^{n-2}$,
\[
(\vect{u} 0 0 \vect{v})
= (\vect{u} (1 0^{n-1}) 0 \vect{v})
= (\vect{u} 1 (0^n) \vect{v})
= (\vect{u} 1 1 \vect{v}).
\]
Similarly as in Case 1.2.1.2 we conclude that the value of $(\cdot)$ at $\vect{a}$ depends only on the parity of the number of $1$'s in $\vect{a}$.
Since $(0^n) = 1$ and $(1 0^{n-1}) = 0$, we have
$(\vect{a}) = 0$ if and only if the number of $1$'s in $\vect{a}$ is odd,
in other words, $(\cdot) = \overline{+^{(n)}}$.
\item[Case 2.2]
$(1 0^{n-1}) = 1$.
Then for all $\vect{u} \vect{v} \in \{0,1\}^{n-2}$,
\[
(\vect{u} 1 0 \vect{v})
= (\vect{u} (1 0^{n-1}) 0 \vect{v})
= (\vect{u} 1 (0^n) \vect{v})
= (\vect{u} 1 1 \vect{v}).
\]
Thus $(1^k 0^{n-k}) = (1^{k+1} 0^{n-k-1})$ for $0 < k < n$.
Since $(1 0^{n-1}) = 1$ and $(0^n) = 1$, it follows that $(\cdot) = c^{(n)}_1$.
\qedhere
\end{description}
\end{description}
\end{proof}

\begin{remark}
For $n \geq 2$, the only $n$\hyp{}ary associative operation on $\{0,1\}$ that is not derivable from a binary associative operation is $\overline{+^{(n)}}$ for odd $n$.
\end{remark}

\begin{proposition}
For $n \geq 1$, the only primitive $n$\hyp{}ary associative operations on $\{0,1\}$ are the unary and binary ones and $\overline{+^{(n)}}$ for $n = 2^k + 1$, $k \in \IN$.
\end{proposition}

\begin{proof}
For each binary semigroup operation $\circ$ (see Table~\ref{table:semigroups}) and for each $n \geq 3$, the $n$\hyp{}ary associative operation $\circ^{(n)}$ is obviously derivable from $\circ$ and is hence not primitive.
It remains to consider operations of the form $\overline{+^{(n)}}$.
The operations derivable from $\overline{+^{(m)}}$ are, for $\ell \in \IN$,
\[
(\overline{+^{(m)}})_\ell^{(\ell(m-1) + 1)} =
\begin{cases}
\overline{+^{(\ell(m-1) + 1)}}, & \text{if $\ell$ is odd,} \\
+^{(\ell(m-1) + 1)}, & \text{if $\ell$ is even.}
\end{cases}
\]
It follows that, for $n \geq 2$, $\overline{+^{(n)}}$ is primitive if and only if $n$ is not of the form $\ell(m-1) + 1$ for any odd $\ell > 1$ and for any $m > 1$.
This is equivalent to $n = 2^k + 1$ for some $k \in \IN$.
\qedhere
\end{proof}

\begin{remark}
The solution to problem 7 in the 2018 Mikl\'os Schweitzer competition \cite{MS}
reveals that
the self\hyp{}commuting Boolean functions
are the same as the associative ones with fictitious arguments introduced.
\end{remark}


\section*{Acknowledgments}

The authors would like to thank Robert Baumann, Thomas Quinn\hyp{}Gregson, and Nikolaas Verhulst for inspiring discussions
and the anonymous reviewer for helpful comments.


\end{document}